\documentclass{cedram-aif}

% ...

\newtheorem{problem}{Problem}[section]

% ...

\newcommand{\Card}{{\rm Card}}
\newcommand{\eps}{\varepsilon}
\newcommand{\bs}{\mathbf s}

\title
[On the expansions of real numbers in two integer bases]
{On the expansions of real numbers in two integer bases}

\alttitle 
{Sur le d\'eveloppement des nombres r\'eels en deux bases enti\`eres} 

\author{\firstname{Yann} \lastname{Bugeaud}}

\address{Universit\'e de Strasbourg, CNRS\\ 
IRMA, UMR 7501\\
7 rue Ren\'e Descartes\\
67084 Strasbourg, France}

\email{bugeaud@math.unistra.fr}

\author{\firstname{Dong Han} \lastname{Kim}}

\address{Dongguk University -- Seoul\\ 
Department of Mathematics Education\\
30 Pildong-ro 1-gil, Jung-gu\\
Seoul 04620, Korea}

\email{kim2010@dongguk.edu}

\thanks{Supported by the National Research Foundation of Korea (NRF-2015R1A2A2A01007090).}

  
\altkeywords{Combinatoire des mots, mot sturmien, d\'eveloppement en base enti\`ere, 
fraction continue}


\begin{document}
%% Abstract 
\begin{abstract}
Let $r$ and $s$ be multiplicatively independent positive integers.   
We establish that 
the $r$-ary expansion and the $s$-ary expansion 
of an irrational real number, viewed as infinite words on $\{0, 1, \ldots , r-1\}$
and $\{0, 1, \ldots , s-1\}$, respectively, cannot have simultaneously a low block complexity. 
In particular, they cannot be both Sturmian words.   
\end{abstract}

%% French abstract
\begin{altabstract}
Soient $r$ et $s$ deux entiers strictement positifs multiplicativement ind\'ependants.  
Nous d\'emontrons que les d\'eveloppements en base $r$ et en base $s$ d'un nombre 
irrationnel, vus comme des mots infinis sur les alphabets $\{0, 1, \ldots , r-1\}$
et $\{0, 1, \ldots , s-1\}$, respectivement, ne peuvent pas avoir simultan\'ement une 
trop faible complexit\'e par blocs. 
En particulier, au plus l'un d'eux est un mot sturmien.  
\end{altabstract}

\maketitle

\section{Introduction}

Throughout this paper, $\lfloor x \rfloor$ denotes the greatest
integer less than or equal to $x$ 
and $\lceil x \rceil$
denotes the smallest integer greater than or equal to $x$.
Let $b \ge 2$ be an integer.
For a real number $\xi$, write
$$
\xi = \lfloor \xi \rfloor + \sum_{k \ge 1} \, {a_k \over b^k} = 
\lfloor \xi \rfloor + 0.a_1 a_2 \ldots ,
$$
where each digit $a_k$ is an integer from $\{0, 1, \ldots , b-1\}$ and
infinitely many digits $a_k$ are not equal to $b-1$.  
The sequence ${\bf a} := (a_k)_{k \ge 1}$ is uniquely determined by
the fractional part of $\xi$.
With a slight abuse of notation, we call it 
the $b$-ary expansion of $\xi$ and we view it also as the infinite word 
${\bf a} = a_1 a_2 \ldots $ over the alphabet $\{0, 1, \ldots , b-1\}$. 

For an infinite word ${\bf x} = x_1 x_2 \ldots $ over a finite alphabet 
and for a positive integer $n$, set
$$
p(n, {\bf x}) = \Card\{ x_{j+1} \ldots x_{j+n} : j \ge 0\}. 
$$
This notion from combinatorics on words is now commonly used to
measure the complexity of the $b$-ary expansion of a real number $\xi$. 
Indeed, for a positive integer $n$, we denote by 
$p(n, \xi, b)$ the total number of distinct blocks of $n$ digits
in the $b$-ary expansion ${\bf a}$ of $\xi$, that is,
$$
p(n, \xi, b) := p(n, {\bf a}) = \Card\{ a_{j+1} \ldots  a_{j+n} : j \ge 0\}. 
$$
Obviously, we have
$
1 \le p(n, \xi, b) \le b^n,
$
and both inequalities are sharp. In particular, the entropy of $\xi$
to base $b$, denoted by $E(\xi, b)$ and defined by
(note that the sequence $(\log p(n, \xi, b))_{n \ge 1}$ is sub-additive,  
thus the limit below exists)
$$
E(\xi, b) := \lim_{n \to + \infty} \, \frac{\log p(n, \xi, b)}{n},
$$
satisfies
$$
0 \le E(\xi, b) \le \log b.
$$
If $\xi$ is rational, then its $b$-ary expansion is
ultimately periodic and the numbers $p(n, \xi, b)$, $n \ge 1$, 
are uniformly bounded by a constant depending only 
on $\xi$ and $b$. 
If $\xi$ is irrational, then, 
by a classical result of Morse and Hedlund \cite{MoHe},
we know that $p(n, \xi, b)\ge n+1$ for
every positive integer $n$, and this inequality is sharp.

\begin{defi} 
A Sturmian word $\mathbf x$ is an infinite word 
which satisfies
$$
p(n,{\mathbf x}) = n + 1, \quad \hbox{for $n \ge 1$}.
$$
A quasi-Sturmian word $\mathbf x$ is an infinite word 
which satisfies
$$
p(n,{\mathbf x}) = n + k, \quad \hbox{for $n \ge n_0$},
$$
for some positive integers $k$ and $n_0$. 
\end{defi}

There exist uncountably many Sturmian words over $\{0, 1\}$; see, e.g., \cite{AlSh03}.

The following rather general problem was investigated in \cite{Bu12}. 
Recall that two positive integers $x$ and $y$ are called 
{\it multiplicatively independent} if 
the only pair of integers $(m, n)$ such that
$x^m y^n = 1$ is the pair $(0, 0)$.

\setcounter{problem}{1}
\begin{problem}
Are there irrational real numbers having a `simple'
expansion in two multiplicatively independent bases?
\end{problem}
\setcounter{cdrthm}{2}

Among other results, Theorem 2.3 of \cite{Bu12} asserts that any 
irrational real number cannot have simultaneously 
too many zeros in its $r$-ary expansion and in its $s$-ary expansion
when $r$ and $s$ are multiplicatively independent positive integers. 
Furthermore, by Theorem 2.1 of \cite{Bu12}, there are irrational numbers 
having maximal entropy in no base. More precisely, for any real number $\eps > 0$
and any integer $b_0 \ge 2$, there exist uncountably many real numbers $\xi$ such that
$$
E(\xi, b_0) < \eps \quad \hbox{and} \quad
E(\xi, b) < \log b, \quad \hbox{for $b \ge 2$}.
$$
However, 
we still do not know whether there exist irrational real numbers 
having zero entropy in two multiplicatively independent bases. 
The main purpose of the present work is to make a small step towards the resolution
of this problem, by establishing that the complexity function of the 
$r$-ary expansion of an irrational real number and that of its $s$-ary expansion 
cannot both grow too slowly when $r$ and $s$ 
are multiplicatively independent positive integers. 

\begin{theorem}\label{twobases}
Let $r$ and $s$ be multiplicatively independent positive integers.
Any irrational real number $\xi$ satisfies
$$
\lim_{n \to + \infty} \, \bigl( p(n, \xi, r) + p(n, \xi , s) -  2n \bigr) = + \infty.
$$
Said differently, $\xi$ cannot have simultaneously a quasi-Sturmian 
$r$-ary expansion and a quasi-Sturmian $s$-ary expansion. 
\end{theorem}

Theorem \ref{twobases} answers Problem 3 of \cite{Bu12} and gives the first
contribution to Problem 10.21 of \cite{BuLiv2}. 

At the heart of the proof of Theorem \ref{twobases} lies the rather surprising fact that 
we can obtain very precise information on the continued fraction expansion 
of any real number whose expansion in some integer base is 
given by a quasi-Sturmian word;  
see Theorem \ref{contfracst} below. This fact was already used in the proof of
Theorem 4.5 of \cite{BuKim15a}. 
The proof of Theorem \ref{twobases} also depends on the $S$-unit theorem, 
whose proof ultimately rests on the Schmidt Subspace Theorem.

We complement Theorem~\ref{twobases} by the following statement 
addressing expansions of a real number in two multiplicatively dependent bases.

\begin{theorem}\label{twobasesdepter}    
Let $r, s \ge 2$ be multiplicatively dependent integers and $m, \ell$ be the smallest 
positive integers such that $r^m = s^\ell$.
Then, there exist uncountably many real numbers $\xi$ satisfying  
$$
\lim_{n \to + \infty} \, \bigl( p(n, \xi, r) + p(n, \xi , s) -  2n \bigr) = m + \ell  
$$ 
and every irrational real number $\xi$ satisfies   
$$
\lim_{n \to + \infty} \, \bigl( p(n, \xi, r) + p(n, \xi , s) -  2n \bigr) \ge m + \ell.  
$$ 
\end{theorem}

The proof of Theorem  \ref{twobasesdepter} consists in rather tedious combinatorial   
constructions and is given in \cite{BuKim16}.  

Our paper is organized as follows. 
Section 2 gathers auxiliary results on Sturmian and quasi-Sturmian words.
Theorem \ref{twobases} is proved in Section 3.

%%%%%%%%%%%%%%%%%%%%%%%%%%%%%%%%%%%%%%%%%%%%%%

\section{Auxiliary results}

Our proof makes use of the complexity function studied in \cite{BuKim15a}, 
which involves the smallest return time of a factor of an infinite word. 
For an infinite word ${\mathbf x}= x_1 x_2 \dots $ set 
$$ 
r(n,{\mathbf x}) = \min \{ m \ge 1 :   
x_{i}\dots x_{i+n-1} = x_{m-n+1} \dots x_{m} \text{ for some } 1 \le i \le m-n \}.    
$$
Said differently, $r(n,{\mathbf x})$ denotes the length of the smallest prefix of $\mathbf x$
containing two (possibly overlapping) occurrences of some word of length $n$.

The next lemma gathers several properties of the function $n \mapsto r(n,{\mathbf x})$ 
established in \cite{BuKim15a}.

\begin{lemma}\label{ubound}
Let $\mathbf x$ be an infinite word. \\   
(1) For any positive integer $n$, we have 
$$
r(n+1,{\mathbf x}) \ge r(n,{\mathbf x}) +1 \quad \hbox{and} \quad 
r(n,{\mathbf x}) \le p(n,{\mathbf x}) + n.
$$ 
(2) The word  $\mathbf x$ is ultimately periodic 
if and only if $r(n+1, {\bf x}) = r(n, {\bf x}) + 1$ for every sufficiently large $n$. \\  
(3) If the positive integer $n$ satisfies $r(n+1,{\mathbf x}) \ge r(n,{\mathbf x}) + 2$, 
then $r(n+1,{\mathbf x}) \ge 2n+3.$ \\  
\end{lemma}

We will make use of the following characterisation of quasi-Sturmian words.

\begin{lemma}\label{Cas}
An infinite word ${\mathbf x}$ written over a finite alphabet ${\mathcal A}$ 
is quasi-Sturmian if and only if there are a 
finite word $W$, a Sturmian word $\bs$ defined over $\{0, 1\}$ and a 
morphism $\phi$ from $\{0, 1\}^*$ into ${\mathcal A}^*$ such that 
$\phi (01) \not= \phi (10)$ and 
$$
{\mathbf x} = W \phi (\bs).
$$
\end{lemma}

\begin{proof}
See \cite{Cassa98}.
\end{proof}

Throughout this paper, for a finite word $W$ and an integer $t$, we write $W^t$ 
for the concatenation of $t$ copies of $W$ and $W^{\infty}$ 
for the concatenation of infinitely many copies of $W$. 
We denote by $|W|$ the length of $W$, that is, the number of letters composing $W$. 
A word $U$ is called periodic if $U = W^t$ for some finite word $W$ and an integer $t \ge 2$.
If $U$ is periodic, then the period of $U$ is defined as the length of the shortest 
word $W$ for which there exists an integer $t \ge 2$ such that $U = W^t$.

\begin{lemma}\label{repsturm}
Let $\mathbf s$ be a quasi-Sturmian word and $W$ be a factor of $\mathbf s$.
Then, there exists a positive integer $t$ such that the word $W^t$ is not a factor 
of $\mathbf s$.
\end{lemma}

This result is certainly well-known.  
For the sake of completeness, we provide its proof.    
A factor $U$ is called a right (resp. left) special word of an infinite word $\mathbf x$ 
if there are two distinct letters $a,b$ such that $Ua$, $Ub$ 
(resp. $aU$, $bU$) are both factors of $\mathbf x$.  
If $V$ occurs infinitely many times in an
infinite word $\mathbf x$ which is not ultimately periodic, then there are 
a right special word $U_1$ and a left special word $U_2$  
such that $V$ is a prefix of $U_1$ and a suffix of $U_2$.

\begin{proof} 
Let $\mathbf s$ be a quasi-Sturmian word and $W$ 
a factor of $\mathbf s$ such that $W^t$ is a factor of $\mathbf s$ for any positive integer $t$. 
Since $n \mapsto p(n,\mathbf s)$ is strictly increasing \cite{MoHe}, 
there is an integer $N$ such that $p(n+1,\mathbf s) = p(n,\mathbf s) +1$ for all $n \ge N$.   
Thus, for every $n \ge N$, the word $\mathbf s$ has only one
right (resp. left) special word of length $n$.  
Since $\mathbf s$ is not ultimatly periodic,    
there exist infinitely many integers $u$ 
such that $V W^u V'$ is a factor of $\mathbf s$ for some $V, V' \ne W$ with $|V| = |V'| =|W|$.  
Consequently, there exist words $U, U'$ with $|U'| > |U| \ge \max( N , |W|)$ and 
letters $a,b,c,d, a',b',c',d'$ such that 
$$
aUb, a'U'b' \quad \hbox{are factors of $W^\infty$}, 
\quad  c U d, c' U' d' \quad \hbox{are factors of $\mathbf s$}, 
$$
$$
\hbox{but $a \ne c, b\ne d, a' \ne c', b'\ne d'$.}
$$  
Since $p(n, W^\infty) \le |W|$ for every positive integer $n$, 
the condition $|U'| >|U| \ge |W|$ implies that $cU$, $c'U'$, $Ud$, $U'd'$ 
are not factors of $W^\infty$.   
Furthermore, our assumption that 
$W^t$ is a factor of $\mathbf s$ for any positive integer $t$ implies that 
the words $aUb$ and $a'U'b'$ are factors of $\mathbf s$. 
Thus the words $U$ and $U'$ are right special and left special words of $\mathbf s$.   
Since the right (resp. left) special word of $\mathbf s$ of length $|U|$ 
is unique, we deduce that
$U$ is a prefix and a suffix of $U'$ and we infer that $a= a'$, $b=b'$, $c=c'$, $d=d'$.
Therefore, $cUb$ is a prefix of $c'U'$ and $aUd$ is a suffix of $U'd'$,
which implies that $cU$ and $aU$ are two distinct right special words of
$\mathbf s$ of length greater than $N$, a contradiction. 
\end{proof}

%%%%%%%%%%%%%%%%%%%%%%%%%%%%%%%%%%%%%%%%%%%%%%

\section{Rational approximation to quasi-Sturmian numbers}

To establish Theorem \ref{twobases} we need a precise description of the convergents 
to quasi-Sturmian numbers. 
The first assertion (and even a stronger version) of the 
next theorem has been proved in \cite{BuKim15a}.

\begin{theorem}\label{contfracst}
Let $\xi$ be a real number whose $b$-ary expansion is a quasi-Sturmian word.
There exist infinitely many rational numbers $\frac{p}{q}$ with $q \ge 1$ such that
\begin{equation}\label{eq3.1}
\Bigl| \xi - \frac{p}{q} \Bigr| < \frac{1}{q^{5/2}}.      
\end{equation}
Furthermore, there exists an integer $M$, depending only on $\xi$ and $b$, such that, for every
reduced rational number $\frac{p}{q}$ satisfying 
$$
\Bigl| \xi - \frac{p}{q} \Bigr| < \frac{1}{M q^2},
$$
with $q$ sufficiently large, there exist integers $r \ge 0$, $s \ge 1$ and $m$ with 
$1 \le m \le M$ and $q = \frac{b^r (b^s - 1)}{m}$. 
\end{theorem}

\begin{proof}
For the proof of the second assertion, 
we assume that the reader is familiar with the theory of continued fractions
(see e.g. Section 1.2 of \cite{BuLiv}).  

We may assume that $\lfloor \xi \rfloor = 0$
and write  
$$
\xi = \sum_{k \ge 1} \, \frac{x_k}{b^k} = [0; a_1, a_2, \ldots].
$$ 
Let $(\frac{p_j}{q_j})_{j \ge 1}$ denote the sequence of convergents to $\xi$. 
 
Since $\xi$ is irrational, it follows from Lemma~\ref{ubound} (2) that 
the increasing sequence ${\mathcal N} := (n_k)_{k \ge 1}$ 
of all the integers $n$ such that $r(n+1, {\bf x}) \ge r(n, {\bf x}) + 2$ is an infinite sequence. 

Since ${\bf x}$ is a quasi-Sturmian sequence, there exist integers $n_0$ and $\rho$ 
such that 
$$
p(n, {\bf x}) \le n + \rho, \quad \hbox{for $n \ge 1$,     
with equality for $n \ge n_0$}.
$$
By Lemma~\ref{ubound} (1),   
we deduce that $r(n, {\bf x}) \le 2 n + \rho$ for $n \ge 1$. 

Let $k$ be a positive integer. 
By Lemma~\ref{ubound} (3),   
we have $r(n_k +1, {\bf x}) \ge 2 n_k + 3$.  
Define $\rho_k$ by $r(n_k +1, {\bf x}) = 2 n_k + \rho_k +1$ 
and observe that $2 \le \rho_k \le \rho+1$.
We deduce from the definition of the sequence ${\mathcal N}$ that 
\begin{equation}\label{new}
r(n_k + \ell, {\bf x}) = 2 n_k + \rho_k + \ell, \quad 1 \le \ell \le n_{k+1} - n_k.
\end{equation}

Set $\alpha_k = \frac{r(n_k, {\bf x})}{n_k}$ and observe that $\alpha_k \le 2 + \frac{\rho}{n_k}$.  
It follows from the choice of $n_k$ and Lemma~\ref{ubound} (1) that 
$$
\liminf_{k \to + \infty} \alpha_k = \liminf_{n \to + \infty} \frac{r(n, {\bf x})}{n}, 
$$ 
which is shown to be less than 2 in \cite{BuKim15a}.  
Consequently, there are infinitely many $k$ such that $\alpha_k < 2$. 
Let $k$ be an integer with this property.

By Theorem 1.5.2 of \cite{AlSh03}, the fact that  
$r(n_k, {\bf x}) = \alpha_k n_k < 2 n_k$ implies that there are finite words $W_k, U_k, V_k$ 
of lengths $w_k, u_k, v_k$, respectively, and a positive integer $t$ such that   
$W_k (U_k V_k)^{t+1} U_k$ is a prefix of ${\bf x}$ of length $\alpha_k n_k$ and 
$$
t (u_k + v_k) + u_k = n_k.
$$
Thus, there exists an integer $r_k$ such that 
the $\alpha_k n_k$ first digits of ${\bf x}$ and those 
of the $b$-ary expansion $W_k (U_k V_k)^{\infty}$ 
of the rational number $\frac{r_k}{b^{w_k} (b^{u_k + v_k} - 1)}$ coincide.
Consequently, we get
$$
\Bigl| \xi - \frac{r_k}{b^{w_k} (b^{u_k + v_k} - 1)} \Bigr| \le 
\frac{1}{b^{\alpha_k n_k}}.
$$
Also, since $u_k + v_k + w_k = (\alpha_k - 1) n_k$, we have
$$
b^{w_k} (b^{u_k + v_k} - 1) \le b^{(\alpha_k - 1) n_k}. 
$$
A classical theorem of Legendre (see e.g. Theorem 1.8 of \cite{BuLiv}) 
asserts that, if the irrational real number $\zeta$ 
and the rational number $\frac{p}{q}$ with $q \ge 1$ 
satisfy $|\zeta - \frac{p}{q}| \le \frac{1}{2 q^2}$,
then $\frac{p}{q}$ is a convergent of the continued fraction expansion of $\zeta$.

As
$$
2 \bigl( b^{w_k} (b^{u_k + v_k} - 1) \bigr)^2 \le 
2 b^{2 (\alpha_k - 1) n_k} \le b^{\alpha_k n_k}
$$
holds if $\alpha_k n_k \le 2 n_k - 1$,  
Legendre's theorem and the assumption $\alpha_k < 2$ imply that 
the rational number $\frac{r_k}{b^{w_k} (b^{u_k + v_k} - 1)}$, which may not be 
written under its reduced form,
is a convergent, say $\frac{p_h}{q_h}$, of the continued fraction expansion of $\xi$.

Let $\ell$ be the smallest positive integer such that $\alpha_{k + \ell} < 2$.

We first establish that $\ell \le 2$ if $n_k$ is sufficiently large.

Assume that $r(n_{k+1}, {\mathbf x}) \ge 2 n_{k+1}$ and $r(n_{k+2}, {\mathbf x}) \ge 2 n_{k+2}$.  
Since  
\begin{equation}\label{eq3.2} 
\begin{split}
r(n_{k+2}, {\bf x}) 
& = r(n_{k+1} + (n_{k+2} - n_{k+1}), {\bf x}) \\ 
& = 2n_{k+1} + \rho_{k+1} + n_{k+2} - n_{k+1}  \\
& = n_{k+2}  + n_{k+1} + \rho_{k+1}
\end{split}
\end{equation}
by \eqref{new}, 
we get
$n_{k+2} - n_{k+1} \le \rho_{k+1}$.  
Put $\eta_k := r(n_{k+2}, {\mathbf x}) - r(n_{k+1}, {\mathbf x})$. 
Then, it follows from \eqref{eq3.2} that 
$$
\eta_k \le n_{k+2} + n_{k+1} + \rho_{k+1} - 2 n_{k+1}
=  n_{k+2} - n_{k+1} + \rho_{k+1} \le 2 \rho_{k+1},  
$$
thus $\eta_k \in \{1, 2, \ldots , 2 \rho + 2\}$. 

Since, for any 
sufficiently large $n \ge 1$, we have $p(n+1, {\bf x}) = p(n, {\bf x}) + 1$, there 
exists a unique factor $Z_n$ of ${\bf x}$ of length $n$ such that $Z_n$ is the prefix of  
exactly two distinct factors of ${\bf x}$ of length $n+1$.

It follows from our assumption $r(n_{k+1} + 1, {\mathbf x}) > r(n_{k+1}, {\mathbf x}) + 1$ 
that 
$$
Z_{n_{k+1}} = {x}_{r(n_{k+1}, {\mathbf x})-n_{k+1}+1} \dots x_{r(n_{k+1}, {\mathbf x})}.  
$$
Likewise, we get 
$$
Z_{n_{k+2}} = {x}_{r(n_{k+2}, {\mathbf x})-n_{k+2}+1} \dots x_{r(n_{k+2}, {\mathbf x})}   
$$
and, since $Z_{n_{k+1}}$ is a suffix of $Z_{n_{k+2}}$, we get  
\begin{align*}    
Z_{n_{k+1}} &= { x}_{r(n_{k+1}, {\mathbf x})-n_{k+1}+1}\dots x_{r(n_{k+1}, {\mathbf x})}    
= { x}_{r(n_{k+2}, {\mathbf x})-n_{k+1}+1}\dots x_{r(n_{k+2}, {\mathbf x})}   \\
&= { x}_{r(n_{k+1}, {\mathbf x}) + \eta_k -n_{k+1}+1}\dots x_{r(n_{k+1}, {\mathbf x}) + \eta_k}.   
\end{align*}  
It then follows from Theorem 1.5.2 of \cite{AlSh03} that there exists an integer $t_k$, a
word $T_k$ of length $\eta_k$ and a prefix $T'_k$ of $T_k$ such that
$$
Z_{n_{k+1}} = (T_k)^{t_k} T'_k.   
$$
We deduce that
$$
t_k \ge \frac{n_{k+1} - \eta_k +1}{\eta_k} \ge  \frac{n_{k+1} - 2 \rho - 1}{2 (\rho + 1)}.   
$$

By Lemma \ref{repsturm}, there exists an integer $t$ such that, for every factor $W$ 
of ${\bf x}$ of length at most $2 \rho + 2$, the word $W^t$ is not a factor of ${\bf x}$. 
We conclude that $n_{k+1}$, 
hence $k$, must be bounded.  

Consequently, if $k$ is sufficiently large, then we cannot have simultaneously 
$r(n_{k+1}, {\mathbf x}) \ge 2 n_{k+1}$ and $r(n_{k+2}, {\mathbf x}) \ge 2 n_{k+2}$.  
This implies that $\ell = 1$ or $\ell = 2$.

Since $\alpha_{k + \ell} < 2$, it follows from Legendre's theorem that the rational number 
$\frac{r_{k+\ell}}{b^{w_{k+\ell}} (b^{u_{k+\ell} + v_{k+\ell}} - 1)}$, 
defined analogously as the rational number $\frac{r_k}{b^{w_k} (b^{u_k + v_k} - 1)}$
and which may not be 
written under its reduced form, is a convergent, say $\frac{p_j}{q_j}$, 
of the continued fraction expansion of $\xi$.  

Here, the indices $h$ and $j$ depend on $k$.   
We have 
$$ 
q_h \le b^{w_k} (b^{u_k + v_k} - 1) \le b^{(\alpha_k-1)n_k}, \quad 
q_j \le b^{w_{k+\ell}} (b^{u_{k+\ell} + v_{k+\ell}} - 1) \le b^{(\alpha_{k+\ell}-1)n_{k+\ell}}.    
$$ 
Note that it follows from \eqref{eq3.2} that
 $$
(\alpha_{k+2} - 1) n_{k+2} = r(n_{k+2}, {\bf x})  - n_{k+2} \le n_{k+1} + \rho + 1,  
$$
and, likewise,
$$
(\alpha_{k+1} - 1) n_{k+1} \le n_k + \rho + 1.
$$
In particular, we have $n_{k+1} \le n_k + \rho + 1$ if $\alpha_{k+1} \ge 2$.  

The properties of continued fractions give that
$$
\frac{1}{2 q_h q_{h+1}} < \Bigl| \xi - \frac{p_h}{q_h} \Bigr| < \frac{1}{q_h q_{h+1}}, \quad
\frac{1}{2 q_j q_{j+1}} < \Bigl| \xi - \frac{p_j}{q_j} \Bigr| < \frac{1}{q_j q_{j+1}}.
$$
This implies that  
$$
q_{j+1} > \frac{b^{\alpha_{k+\ell} n_{k+\ell}}}{2 q_j} \ge \frac{b^{n_{k+\ell}}}{2}.
$$
Since $\alpha_k < 2$, we get
$$
q_h \le b^{(\alpha_{k} - 1) n_{k}} <  b^{n_k}  \le \frac{b^{n_{k+\ell}}}{2} < q_{j+1}. 
$$
Combined with $\frac{p_h}{q_h} \not=  \frac{p_j}{q_j}$,
this gives 
$$
q_h < q_{h+1} \le q_j < q_{j+1}.
$$

It follows from
$$
q_{h} > \frac{b^{\alpha_k n_k}}{2 q_{h+1}}
$$
and
$$
q_{h+1} \le q_j \le b^{(\alpha_{k+\ell} - 1) n_{k+\ell}} \le b^{n_k + 2(\rho +1)},
$$
that 
$$
q_h >  \frac{b^{\alpha_k n_k}}{2 b^{n_k + 2(\rho +1)}} = \frac{b^{(\alpha_k - 1) n_k}}{2 b^{2(\rho +1)}}. 
$$
Since $q_h \le b^{w_k} (b^{u_k + v_k} - 1) \le b^{(\alpha_{k} - 1) n_{k}}$,  
this shows that the rational number 
$\frac{r_k}{b^{w_k} (b^{u_k + v_k} - 1)}$
is not far from being reduced, 
in the sense that the greatest common divisor of its numerator and denominator 
is at most equal to $2 b^{2(\rho +1)}$. 
Furthermore, it follows from
$$
q_{h+1} > \frac{b^{\alpha_k n_k}}{2 q_h} \ge \frac{b^{n_k}}{2}
$$
that
$$
1 \le \frac{q_j}{q_{h+1}} \le 2 b^{2(\rho +1)}.
$$
Consequently, all the partial quotients $a_{h+2}, \ldots , a_j$  are less than $2 b^{2(\rho +1)}$
and we get 
$$
\Bigl| \xi - \frac{p_\ell}{q_\ell} \Bigr| > \frac{1}{(a_{\ell + 1} + 2) q_{\ell}^2} \ge
\frac{1}{2 (b^{2(\rho +1)} + 1) q_{\ell}^2}, \quad
\hbox{for $\ell = h+1, \ldots , j-1$}.
$$
Consequently, the second assertion of the theorem holds with 
the value $M = 2 (b^{2(\rho +1)} + 1)$. 
\end{proof}

\medskip

\noindent {\it Proof of Theorem \ref{twobases}.}

\medskip

Let $\xi$ be a real number whose $r$-ary expansion and whose 
$s$-ary expansions are quasi-Sturmian words. 
Since \eqref{eq3.1} has infinitely many solutions and 
by unicity of the continued fraction expansion of an irrational real number,  
we deduce from Theorem \ref{contfracst} that there are positive integers $m_1, m_2$ 
and infinitely many quadruples $(u_1, v_1, u_2, v_2)$ of non-negative integers 
with $v_1 v_2 \not= 0$ such that
$$
\frac{r^{u_1} (r^{v_1} - 1)}{m_1} = \frac{s^{u_2} (s^{v_2} - 1)}{m_2}.
$$
In particular, the equation
\begin{equation}\label{eq3.3}
\frac{m_2}{m_1} r^{z_1} s^{-z_4} - \frac{m_2}{m_1}  r^{z_2} s^{-z_4} + s^{z_3} = 1  
\end{equation}
has infinitely many solutions $(z_1, \ldots , z_4)$ in non-negative integers with $z_3 \ge 1$. 

This is a linear equation in variables which lie in a 
multiplicative group generated by $r$ and $s$. 
By Theorem 1.1 of \cite{ESS}, Equation \eqref{eq3.3} has only finitely many non-degenerate 
solutions. 
Since $r$ and $s$ are multiplicatively independent, the equation 
$\frac{m_2}{m_1}  r^{z_2} s^{-z_4} =  s^{z_3}$  
has only finitely many solutions. 
Consequently, Equation \eqref{eq3.3} has only finitely many solutions, a contradiction to our
assumption. Consequently, the $r$-ary and the $s$-ary expansions of $\xi$ cannot 
both be quasi-Sturmian.

\section*{Acknowledgement}
The authors are grateful to the referee for a very careful reading. 

\nocite{*}
\bibliographystyle{cdraifplain}
%\bibliography{xampl}

\end{document}